\documentclass{gtart}

\usepackage{psfrag, graphicx, subfigure, epsfig,color}

\usepackage{amsmath, amssymb, latexsym, euscript}

\usepackage{mathptmx}

\usepackage{amsthm} 
\usepackage{color}

\definecolor{refkey}{rgb}{1,0,1} 
\definecolor{labelkey}{rgb}{1,0,1}

\usepackage{graphicx,epsfig}
\usepackage{epstopdf}

\usepackage{url}
\usepackage[all]{xy}
\usepackage{psfrag}
\usepackage{hyperref} 

\newtheorem{theorem}{Theorem}[section]
\newtheorem{lemma}[theorem]{Lemma}
\newtheorem{corollary}[theorem]{Corollary}

\newtheorem*{theorem*}{Theorem}

\def\gap{\vspace{.3cm}\noindent}

\def\smallskip{\vspace{.15cm}}
\def\medskip{\vspace{.3cm}}
\def\text{\mbox}
\renewcommand \P{\mathsf P}
\def\Lf{\mathcal L}
\def\Lf{\mathcal L} 
\newcommand{\R}{\mathbb{R}}
\newcommand{\C}{\mathbb{C}}
\newcommand{\rpn}{\R\P^n}
\def\blue{\color{blue}}

\newcommand{\trace}{\mathsf{trace}}

\newcommand{\pgln}{\mathsf{PGL(n+1,\mathbb{R}})}
\newcommand{\pglf}{\mathsf{PGL(4,\mathbb{R}})}
\newcommand{\glf}{\mathsf{GL(4,\mathbb{R}})}
\newcommand{\slf}{\mathsf{SL(4,\mathbb{R}})}

\newcommand{\rpth}{\mathbb{R}\mathsf{P}^3}
\newcommand{\rpt}{\mathbb{R}\mathsf{P}^2}
\newcommand{\dev}{\mathsf{dev}}
\newcommand{\hol}{\mathsf{hol}}
\newcommand{\Isom}{\mathsf{Isom}}
\newcommand{\Hom}{\mathsf{Hom}}
\newcommand{\cpo}{\C\P^1}

\begin{document}

\title{A 3-Manifold with no Real Projective Structure.}
\author{Daryl Cooper and William Goldman}

\begin{abstract}We show that the connected 
sum of two copies of real 
projective $3$-space does not admit a real projective structure. This 
is the first known example of a connected 3-manifold without 
a real projective structure.
\end{abstract}

\primaryclass{57M25, 57N10}
\keywords{3--manifold, projective geometry, Thurston}

\maketitle
\bigskip\centerline{\textit{Dedicated to Michel Boileau on the occasion of his sixtieth birthday.}} 

\tableofcontents

\section{Introduction}
Geometric structures modeled on homogeneous spaces of Lie groups were introduced
by Ehresmann~\cite{Ehresmann}. If $X$ is a manifold upon which a Lie group $G$ acts
transitively, then an {\em Ehresmann structure\/} modeled on the homogeneous space $(G,X)$
is defined by an atlas of coordinate charts into $X$ such that the coordinate changes
locally lie in $G$. For example, an Ehresmann structure modeled on Euclidean geometry
is equivalent to a flat Riemannian metric. More generally, constant curvature Riemannian
metrics are Ehresmann structures modeled on the sphere or hyperbolic space and their
respective groups of isometries.  
A recent survey of the theory of Ehresmann structures on low-dimensional
manifolds is \cite{G4}.
Ehresmann $(G,X)$-structures
are special cases of flat Cartan connections (modeled on $(G,X)$) with vanishing
curvature. See Sharpe~\cite{Sharpe} for a modern treatment of this theory.

{\em Topological uniformization \/} in dimension $2$ asserts that
every closed $2$-manifold admits a constant curvature Riemannian
metric. Therefore every such surface is uniformized by one of three
Ehresmann structures corresponding to constant curvature Riemannian
geometry.  However, projective and conformal geometry provide two
{\em larger geometries,\/} each of which uniformize {\em all\/} surfaces
(Ehresmann~\cite{Ehresmann}). 

 
The subject received renewed attention in the late 1970's by W.\ Thurston, who cast his 
Geometrization Conjecture (now proved by Perelman) in terms of Ehresmann $(G,X)$-structures. 
Thurston proposed that the relevant geometries are 
{\em locally homogeneous $3$-dimensional Riemannian manifolds.\/} 
These are the $3$-dimensional homogeneous spaces $G/H$ where the 
isotropy group $H$ is compact. 
Up to local isometry, those which cover compact $3$-manifolds fall into eight types.
See Scott~\cite{S},  Thurston~\cite{Thurston} and 
Bonahon \cite{Bonahon} for a description of these geometries. 
Every closed $3$-manifold  canonically decomposes along essential elliptic or 
Euclidean $2$-manifolds into pieces, 
each of which admit a geometric structure of one of these eight types.

Since these eight geometries often themselves admit geometric structures modeled on 
homogeneous spaces with {\em noncompact\/} isotropy group, it is tempting to
search for geometries which uniformize {\em every\/} closed $3$-manifold.
\cite{G2} exhibits examples of closed $3$-manifolds which 
admit no flat conformal structures.
(\cite{G2} also contains examples of $3$-manifolds, such as the
$3$-torus,  which admit no spherical CR-structure.)
The purpose of this note is to exhibit a closed $3$-manifold (namely the connected
sum $\rpth\#\rpth$) which does not admit a flat {\em projective\/} structure.
(On the other hand $\rpth\#\rpth$ does admit a flat conformal and spherical CR structures.)

A {\em $\rpn$-structure\/} on a connected smooth n-manifold $M$ is a
Ehresmann structure modeled on $\rpn$ with coordinate changes locally
in the group $\pgln$ of {\em collineations\/} (projective transformations) of $\rpn$.
Such a structure is defined by
an atlas for $M$ where the transition maps are the restrictions of 
projective transformations to open subset of projective n-space. 
Fix a universal covering space $\tilde{M}\to M$; 
then an atlas as above  determines
an immersion called the {\em developing map\/}
\[ 
\tilde{M}\xrightarrow{~\dev~}\rpn \] 
and a homomorphism called the 
{\em holonomy:}
\[ 
\pi_1M\xrightarrow{~\hol_M~}\pgln \] 
such that for all $\tilde{m}\in\tilde{M}$ and all $g\in\pi_1M$ that
\[\dev_M(g\cdot\tilde{m})\ =\ \hol_M(g)\cdot \dev_M(\tilde{m}).\]

Basic questions include the existence and classification 
of $\rpth$-structures on a given $3$-manifold.

Recent progress on classification is documented in 
\cite{CLT1},\cite{CLT2}:
in particular certain closed hyperbolic 3-manifolds admit
continuous families of projective structures containing the
hyperbolic structure, while others do not.

Every 2-manifold $\Sigma$ 
admits a projective structure. 
The convex ones form a cell of 
dimension $16\mathsf{genus}(\Sigma)$
(Goldman~\cite{G3}).
Suhyoung Choi~\cite{Choi} showed that every $\rpt$-manifold of genus 
$g>1$ decomposes naturally into convex subsurfaces.
Combining these two results completely classify
$\rpt$-structures \cite{CG1},\cite{CG2}.  
Almost all geometric
3-manifolds admit a projective structure, in fact:

\begin{theorem*} 
Suppose that $M$ is a 3-manifold equipped with one of the eight Thurston 
geometric structures. 
Then either $M$ is a Seifert fiber space with a fibration that does
not admit an orientation (and there is a double cover which is real
projective) or else $M$ inherits a uniquely determined real projective
structure underlying the given Thurston geometric structure.
\end{theorem*}

All this was  presumably known to Thurston, and was documented by
Thiel\cite{T} and Molnar \cite{M}.  This theorem is a consequence of
the existence of a representation of each of the eight Thurston
geometries $(G,X)$ into $(\rpth,\pglf)$ except that in the case of
the product geometries $S^2\times{\mathbb R}$ and ${\mathbb H}^2\times
{\mathbb R}$ the group $G=Isom(X)$ is replaced by the index-2 subgroup
$\Isom_+(X)$, which preserves the orientation on the ${\mathbb R}$
direction.  In general some 3-manifolds admit a real projective
structure that is not obtained from a Thurston geometric structure
(Benoist~\cite{Benoist3}).  Furthermore exceptional fibered examples
admit exotic real projective structures which do not arise from a
projective representation of the associated geometry.  (Compare
Guichard-Wienhard~\cite{GuichardWienhard} for some examples on twisted
$S^1$-bundles over closed hyperbolic surfaces.)

The manifold $\rpth\# \rpth$ admits a geometric structure modeled on $S^2\times {\mathbb R}$.
Our main result is:
\begin{theorem*} The 3-manifold $M\ =\ \rpth\# \rpth$ does not admit an
$\rpth$-structure. 
\end{theorem*}

One impetus to prove this result is the fact that almost all geometric
3-manifolds in the sense of Thurston have 
projective structures. 
This suggested that such structures might be universal for
3-manifolds, an outcome that would have had significant 
consequences since, for example,  every closed simply connected projective manifold 
is a sphere. (This would imply the Poincar\'e conjecture.)
One could imagine a functional, analogous to the Cherns-Simon invariant
(whose critical points are conformally flat metrics) or the projective Weyl
tensor, whose gradient flow would converge to a flat projective structure.
Instead, as our simple example shows, 
the situation turns out to be more intriguing and complex.

After proving this result we learned from Yves Benoist that this
result can also be deduced from his classification~\cite{Benoist1,Benoist2} of real projective
manifolds with abelian holonomy. However we believe that our direct proof, without using
Benoist's general machinery, may suggest generalizations.
A key point in Benoist's classification 
(see \cite{Benoist2}, \S4.4 and Proposition~4.9 in particular)
 is that the developing image of such an $\rpth$-structure is the complement of a disjoint union of projective subspaces of dimension $0$ (a point) and $2$, presenting a basic asymmetry which is incompatible with the deck transformation of 
$\hat M$. The is is impossible as described in the next paragraph.
The example in  \S\ref{sec:DihedralExample} is a projective 3-manifold 
whose {\em holonomy} is infinite dihedral but not injective. 
In this case the developing image is the complement of two projective lines,
which deformation retracts to a two-torus, and has trivial holonomy. In 
{\em some sense}\footnote{A phrase the first author {\em learned} from Michel Boileau}
this manifold is trying to be $\rpth\# \rpth$.

To give some intuition for the following proof we first show that the
developing map for a real projective structure on $\rpth\#\rpth$ cannot be injective. 
The universal cover of $M$ is $S^2\times{\mathbb R}.$ 
If the developing map embeds this in $\rpth$ then there are two complementary
components and they have the homotopy type of a point and ${\mathbb P}^2.$ 
There is a covering transformation of the universal cover which swaps the ends. 
The holonomy leaves the image of the developing map invariant but swaps the complementary components. 
This is of course impossible since they have different homotopy types. 
Unfortunately one can't in general assume the developing map for a projective structure
is injective.

Currently, it seems to be very difficult to show that a 3-manifold does not admit a projective structure.
We do not know if a connectec sum can ever admit a projective structure. 
Is there a projective structure on a closed Seifert-fibered manifold $\ne S^3$ for which the holonomy of the fiber is trivial?
In this regard, we note that Carri\`ere-d'Albo-Meignez~\cite{MR1233496} 
have shown that several closed Seifert $3$-manifolds do not admit affine structures.

This work was partially supported by NSF grants DMS-0706887,  070781, 1065939, 1207068,  1045292 and 0405605. Furthermore we are grateful for the Research Network in the Mathematical Sciences grant for the GEAR Research Network (DMS-1107367) for partial support as well as the Focused Research Grant DMS-1065965. {\blue Le premier auteur a  appris beaucoup math\'{e}matiques \`{a} partir de Michel et aussi une appr\'{e}ciation de fromage de ch\'{e}vre et du vin ros\'{e}.}

\section{The Ehresmann-Weil-Thurston principle}
Fundamental in the deformation theory of locally homogeneous
(Ehresmann) structures is the following principle,
first observed by Thurston~\cite{Thurston_1979}: 

\begin{theorem}\label{deformations} 
Let $X$ be a manifold upon which a Lie group $G$ acts
transitively. 
Let $M$ have a geometric structure modeled on $(G,X)$
with holonomy representation 
$\pi_1(M)\xrightarrow\rho G$.
For $\rho'$ sufficiently near $\rho$ in the space
of representations $\Hom(\pi_1(M),G)$,
there exists a (nearby) $(G,X)$-structure on $M$ with
holonomy representation $\rho'$.
\end{theorem}
\begin{corollary} 
Let $M$ be a closed manifold. 
The set of holonomy representations of $(G,X)$-structures
on $M$ is open in 
$\Hom(\pi_1(M),G)$.
\end{corollary} 

This principle has a long history.
In the context of $\cpo$-structures, 
this is due to Hejhal~\cite{Hejhal};
see also Earle~\cite{Earle} and Hubbard~\cite{Hubbard}.
The first application is the theorem of Weil~\cite{Weil}
that the set of Fuchsian representations of the fundamental
group of a closed surface group in $\mathsf{PSL}(2,\R)$
is open. The first detailed proofs of this fact are Lok~\cite{Lok},
Canary-Epstein-Green~\cite{CEG}, and Goldman~\cite{G1}
(the proof in \cite{G1} was worked out with M.\ Hirsch,
and were independently found by A.\ Haefliger).
The ideas in these proofs may be traced to Ehresmann.
For a more recent proof, with applications to rigidity, see
Bergeron-Gelander~\cite{BG}.

 In the sequel $M=\rpth\#\rpth$. By Van Kampen's theorem,
$$
\pi_1M\ \cong \langle \ a,b\ :\ a^2=1=b^2\ \rangle
$$ 
is isomorphic to the infinite dihedral group.

\section{An example with dihedral holonomy}\label{sec:DihedralExample}
Although we prove that no $\rpth$-structure exists on $\rpth\#\rpth$,
there do exist $\rpth$-manifolds whose holonomy is the infinite
dihedral group. Namely, consider two linked projective lines
$\ell_1,\ell_2$ in $\rpth$ and a collineation $\gamma$ having $\ell_1$
as a sink and $\ell_2$ as a source. Then the complement
$$
\Omega:= \rpth\setminus (\ell_1\cup\ell_2)
$$ 
is fibered by $2$-tori
and the region between two of them forms a fundamental domain
for the cyclic group $\langle\gamma\rangle$ acting on 
$\Omega$. The quotient $\Omega/\langle\gamma\rangle$ is
an $\rpth$-manifold diffeomorphic to a $3$-torus having cyclic holonomy
group.

Now choose an free involution $\iota$ of $\rpth$ which interchanges
$\ell_1$ and $\ell_2$, conjugating $\gamma$ to $\gamma^{-1}$.  The
group $\Gamma:=\langle\gamma,\iota\rangle$ acts properly and freely on
$\Omega$ and contains the cyclic subgroup $\langle\gamma\rangle$ with
index two. The quotient $\Omega/\Gamma$ is an $\rpth$-manifold with
cyclic holonomy. It is a Bieberbach manifold, having a 
{\em Euclidean  structure.}

In coordinates we may take $\ell_1$ and $\ell_2$ to be the
projectivizations of the linear subspaces $\R^2\times\{0\}$ and
$\{0\}\times\R^2$ respectively. The projective transformations
$\gamma$ and $\iota$ are represented by the respective matrices:
$$
\gamma \longleftrightarrow \bmatrix 
\lambda & 0 & 0 & 0 \\ 0 & \lambda & 0 & 0  \\
0 & 0 & \lambda^{-1} & 0 \\ 0 & 0 & 0 & \lambda^{-1}
\endbmatrix, \quad
\iota \longleftrightarrow \bmatrix 
0 & 0 & -1 & 0  \\ 0 & 0 & 0 & -1  \\
1 & 0 &  0 & 0  \\ 0 & 1 & 0 & 0 
\endbmatrix
$$
 



\section{Proof of Main Theorem}

Using the presentation of $\pi_1M$ above
there is a 
short exact sequence
$$
1\longrightarrow{\mathbb Z}\longrightarrow
\pi_1M\ \cong\ {\mathbb Z}_2 * {\mathbb Z}_2\longrightarrow {\mathbb 
Z}_2\longrightarrow 1.
$$ 
and the product
$c: =  a b$ 
generates the infinite cyclic normal subgroup.  
Corresponding to the subgroup of $\pi_1M$ generated 
by $a$ and $c^n$ there is an  $n$-fold covering space
$M^{(n)}\rightarrow M$. The manifold $M^{(n)}$ is homeomorphic to $M$.
When $n=2$ the cover is regular and
corresponds to the subgroup generated by $a$ and $bab^{-1}.$ Thus any projective structure on $M$ 
yields other projective structures
on (covers of) $M$ whose holonomy has certain desirable properties. We use this trick throughout the paper.

If $M$ admits an $\rpth$ structure, then there is a developing map
$\dev_M:\tilde{M}\rightarrow \rpth$ with holonomy
$\hol_M:\pi_1(M)\rightarrow \pglf$.  Choose $A,B\in \glf$ with
$[A]=\hol_M(a),\ [B]=\hol_M(b)$ .  Set $C=AB.$

In view of the previous remarks,  after passing to the double
covering-space  $M^{(2)}$, there is a projective structure with the
matrices $A$ and $B$ conjugate.  This property continues
to hold after passing to a further $n$-fold
covering space $M^{(2n)}\longrightarrow M^{(2)}$, thereby replacing
$C^2$ by $C^{2n}$. 
This covering, combined with a small deformation,
enables one to reduce the problem to a restricted class of holonomies.

{\em Outline proof.} If $M$ admits a projective structure then
after a small deformation some finite covering is $N=S^2\times
S^1$ with a projective structure with holonomy contained in a
one-parameter group $G$ that becomes diagonal after conjugacy. Furthermore
there is an involution, $\tau$, of $N$ reversing the $S^1$
factor which is realized by a projective map which normalizes $G$.
The flow generated by $G$ on $\rpth$ pulls back to $N$. The
flow on $\rpth$ has stationary points consisting of certain
projective subspaces corresponding to the eigenspaces of $G$. One
quickly reduces to the case that the flow on $N$ is periodic
giving a product structure. The orbit space is $S^2$. The orbit space
of the flow on $\rpth$ is a non-Hausdorff surface ${\mathcal L}$.
The developing map induces an immersion of $S^2$ into ${\mathcal L}$.
There are only two possibilities for ${\mathcal L}$ corresponding to
the two structures of the stationary set. The possibilities for
immersions of $S^2$ into ${\mathcal L}$ are determined. None of these
is compatible with the action of $\tau$. This contradicts the
existence of a developing map.\qed

\begin{lemma}\label{lem1}
The holonomy is injective.
\end{lemma}
\begin{proof} 
Otherwise the holonomy has
image a proper quotient of the infinite dihedral group which is
therefore a finite group.  The cover $\tilde{M}'\rightarrow M$
corresponding to the kernel of the holonomy is then a finite cover
which is immersed into $\rpth$ by the developing map. 
Since $\tilde{M}'$ is compact $\dev$ is a covering map.
Hence $\tilde{M}'$ is a covering-space of $\rpth$. 
But $\pi_1\tilde{M}'$ is infinite, 
which contradicts that it is isomorphic to a subgroup of $\pi_1\rpth\cong{\mathbb Z}_2.$ 
\end{proof}
\medskip

Observe that in $\pi_1M$ that $c$ is conjugate to $c^{-1}$ since
$$c^{-1}=(ab)^{-1}=b^{-1}a^{-1}=ba=b(ab)b^{-1}=bcb^{-1}.$$ It follows
that for each eigenvalue $\lambda$ of $C$ the multiplicity of
$\lambda$ is the same as that of $\lambda^{-1}.$ 

\begin{lemma}\label{lem:diagonalizable}
 We may assume $C$ is diagonalizable over $\mathbb R$ and 
has positive eigenvalues.
\end{lemma}

\begin{proof} After passing to the double cover of $M$ discussed above we may assume
that $A$ and $B$ are conjugate. Since $[A]^2\in \pglf$ is the identity
it follows that after rescaling $A$ we have $A^2=\pm Id,$ thus $A$ is
diagonalizable over ${\mathbb C}$. If $A^2 = Id$ then $A$ has eigenvalues $\pm1.$ Since
we are only interested in $[A]$ we may multiply $A$ by $-1$ and
arrange that the eigenvalue $-1$ has multiplicity at most $2.$
Otherwise $A^2 = -Id$ and $A$ has eigenvalues eigenvalues $\pm i$ each
with multiplicity two.  Thus $A$ is {\em conjugate} in $\glf$ to one of the
matrices:

$$A_1\ =\ \left[\begin{array}{cccc}
0 & 1 & 0 & 0 \\
-1 & 0 & 0 & 0\\
0 & 0 & 0 & 1\\
0 & 0 & -1 & 0\\
\end{array}\right],
\qquad
  A_2\ =\ \left[\begin{array}{cccc}
-1 & 0 & 0 & 0 \\
0 & 1 & 0 & 0\\
0 & 0 & 1 & 0\\
0 & 0 & 0 & 1\\
\end{array}\right],
\qquad
A_3\ =\ \left[\begin{array}{cccc}
-1 & 0 & 0 & 0 \\
0 & -1 & 0 & 0\\
0 & 0 & 1 & 0\\
0 & 0 & 0 & 1\\
\end{array}\right]$$
After conjugating $\hol$ we may further assume that $A=A_i$ for
some $i\in\{1,2,3\}.$ Since $A$ and $B$ are conjugate there is $P\in 
\glf$
such that $B=P\cdot A\cdot P^{-1}.$ Then $C=A\cdot P\cdot A\cdot P^{-1}$.
Changing $P$ is a way to deform $\hol$. The first step is  to show that
when $P$ is in the complement of a certain algebraic subset then $C$
has four distinct eigenvalues and is therefore diagonalizable over ${\mathbb C}$.

Given a homomorphism
$\hol':\pi_1M\rightarrow \pglf$ sufficiently close to $\hol$ by
\ref{deformations} there is a projective structure on $M$ with this 
holonomy.   Consider
the map
$$f:\glf\longrightarrow \slf$$given by
$$f(P) = A\cdot P\cdot A\cdot P^{-1}.$$
This is a regular map defined on  
$\mathsf{GL}(4,\R)$
.
Define $g:\slf\rightarrow{\mathbb R}^2$ 
by $g(Q)=(\trace(Q),\trace(Q^2)).$ This is also a 
regular map.

{\bf Case 1} $A=A_1$ or $A_3$. An easy computation shows 
that the image of $g\circ f$ contains an open set: 
 $$
 \begin{array}{lcc}
 A & P & g\circ f \\
 A_1 &  \left[\begin{array}{cccc}
 0& 1& 0& 0 \\ y& 0& 0& 0 \\ 1& 0& 0& x\\ 0& 0& 1& 1\end{array}\right]
&
   x^{-2}y^{-2} (x^2y + 2x y^2 + x^3 y^2 + x^2 y^3,\   x^2 + 4 y^2 + 2 x^2 y^2 + x^4 y^2 + x^2 y^4)\\
 A_3 & \left[\begin{array}{cccc}
  0 & 0& 1& 0 \\ 1& 0& 0& 1 \\ y+x & 0& 0& y-x \\ 0& 1& 0& 0 \end{array}\right]
 &
    x^{-2}(-2x^2-2xy,\   4y^2)
 \end{array}$$
 
The subset $E\subset \glf$ consisting of all $P$
for which $C=f(P)$ has a repeated eigenvalue
is the affine algebraic set where the discriminant of the characteristic
polynomial of $C$ vanishes. 

{\bf Subclaim} 
$E$ is a proper subset.
\gap

Let $S$ be the set of eigenvalues of $C$. The map
$\tau(z)=z^{-1}$ is an involution on $S$ because $C$ is conjugate to $C^{-1}$. 
Each orbit in $S$ under this involution
contains at most $2$ elements. An orbit of size one consists of either $1$ or $-1$, from which
it follows that
if $P\in E$ then $|S|<4$ and $S\subset\{\pm1,\lambda^{\pm1}\}.$
Thus if $P\in E$  
either $S\subset\{\pm1\}$  or 
$$
\trace(C)=\lambda+\lambda^{-1}+m\quad or\quad \trace(C)=2\lambda+2\lambda^{-1}
$$ 
where $m\in\{0,\pm2\}$ and
$$
\trace(C^2)=\lambda^2+\lambda^{-2}+2.
$$
In each case 
 $\trace(C)$ and $\trace(C^2)$ satisfies an algebraic relation.
Thus
 $\dim [g\circ f(E)]=1$. 
  The image of $g\circ f$ contains an open set therefore
  $E$ is a proper subset, proving the subclaim.\qed

Since $E$ is an algebraic subset of $\glf$ which is a proper subset
 it follows that  $\glf\setminus E$ is open and dense
in the Euclidean topology.  
Hence there is a small perturbation of $P$ and of $\hol$
so that $C$ is diagonalizable over $\C$ and has 4 distinct eigenvalues $\{\lambda_1^{\pm1},\lambda_2^{\pm1}\}$.

By suitable choice of $P$, we can arrange that  
the arguments of $\lambda_1$ and $\lambda_2$ are rational multiples 
of $\pi.$ 
Furthermore passing to a finite covering-space of $M$,
we may assume all eigenvalues of $C$ are real.
Passing to a double covering-space we 
may assume these eigenvalues are positive. However it is possible that
they are no longer distinct.

We have shown in this case that the 
projective structure on (a finite cover of) $M$ may be chosen so that
$C$ is diagonal with real positive eigenvalues, which completes {\bf case 1}.\qed

{\bf Case 2} $A=A_2.$ Then for every choice of
$P$ the $+1$ eigenspaces of $A$ and $B$ intersect in a subspace of
dimension at least $2.$ Since $C=AB$ it follows that there is a
2-dimensional subspace on which $C$ is the identity, and thus $C$ has
eigenvalue $1$ with multiplicity at least $2.$ It is easy to see that
$\trace\circ f$ is not constant,   for example when
$$P\ =\ \left[\begin{array}{cccc}
1 & 0 & 0 & 1 \\
0 & 1 & 1 & 0\\
1 & 1 & 0 & 0\\
0 & x & 0 & 1\\
\end{array}\right]\qquad\qquad \trace(f(P))=4/(1+x)$$

Thus on a dense open set $f(P)\ne 4$ so $C$ has an eigenvalue $\lambda\ne 1$.
 As before, by replacing $C$ by $C^2$ if needed,
we may assume $\lambda\ne \pm1$. Thus
$\lambda^{-1}\ne\lambda$ is also an eigenvalue giving 3 distinct
eigenvalues $\lambda,\lambda^{-1},1,1$. 
Since the $+1$-eigenspace of $C$ has dimension two,
$C$ is diagonalizable over ${\mathbb C}.$ 
The rest of the argument is as before.\end{proof}

\begin{lemma}
We may assume that $C$ is one of the following matrices with
$\lambda_2>\lambda_1>1.$
$$C_1\ =\ \left[\begin{array}{cccc}
\lambda_1 & 0 & 0 & 0 \\
0 & \lambda_1 & 0 & 0\\
0 & 0 & \lambda_1^{-1} & 0\\
0 & 0 & 0 & \lambda_1^{-1}\\
\end{array}\right],
\quad
C_2\ =\ \left[\begin{array}{cccc}
\lambda_1 & 0 & 0 & 0 \\
0 & \lambda_1^{-1} & 0 & 0\\
0 & 0 & 1 & 0\\
0 & 0 & 0 & 1\\
\end{array}\right],
\quad
C_3\ =\ \left[\begin{array}{cccc}
\lambda_1 & 0 & 0 & 0 \\
0 & \lambda_2 & 0 & 0\\
0 & 0 & \lambda_1^{-1} & 0\\
0 & 0 & 0 & \lambda_2^{-1}\\
\end{array}\right]
$$
\end{lemma}
\begin{proof}  
The result follows from 
Lemma~\ref{lem:diagonalizable}
and the fact that $C$ is conjugate to $C^{-1}$. Observe that Lemma~\ref{lem1} rules out all eigenvalues are $1$.
\end{proof}

There is a  1-parameter diagonal subgroup $g:{\mathbb
R}\rightarrow G\subset \pglf$ such that
$g(1) = [C].$ For example if $C=C_3$ then this subgroup is:

$$
g_1(t)\ =\ \left[
\begin{array}{cccc}
\exp(\ell_1 t) & 0 & 0 & 0 \\
0 & \exp(\ell_2 t) & 0 & 0\\
0 & 0 & \exp(-\ell_1 t) & 0\\
0 & 0 & 0 & \exp(-\ell_2 t)\\
\end{array}\right]
\qquad\qquad \ell_i\ =\ \log(\lambda_i).
$$

This group $G$ is characterized as the unique one-parameter subgroup 
which contains the cyclic group $H$ generated by $C$ and such that 
every element in $G$ has real eigenvalues. Since $H$ is normal in 
$\hol(\pi_1M)$ it follows from the characterization that $G$ is 
normalized by $\hol(\pi_1M).$

Let $N\rightarrow M$ be the double cover corresponding to the subgroup
of $\pi_1M$ generated by $c.$ Observe that $N\cong S^2\times S^1.$ Let
$\pi:\tilde{N}\rightarrow N$ be the universal cover. Then $N$ inherits
a projective structure from $M$ with the same developing map
$\dev_N=\dev_M.$ The image of the holonomy for this projective structure
on $N$ is generated by $[C].$ Let $z\in{\it gl}(4,{\mathbb R})$ be an
infinitesimal generator of $G$ so that $G=\exp({\mathbb R}\cdot z).$ Thus
for $C_3$ we have

$$
z\ =\ \ \left[\begin{array}{cccc}
\ell_1 & 0 & 0 & 0 \\
0 & \ell_2 & 0 & 0\\
0 & 0 & -\ell_1 & 0\\
0 & 0 & 0 & -\ell_2\\
\end{array}\right].
$$

There is a flow $\Phi:\rpth\times{\mathbb R}\rightarrow \rpth$ on $\rpth$
generated by $G$ given by $$\Phi(x,t)\ =\ \exp(tz)\cdot x.$$ Let $V$ be
the vector field on $\rpth$ velocity of this flow. The fixed points of the
flow are the zeroes of this vector field. The vector field is
preserved by the flow, and thus by $\hol(\pi_1N).$ It follows that $V$
pulls back via the developing map to a vector field $\tilde{v}$ on
$\tilde{N}$ which is invariant under covering transformations and thus
covers a vector field $v$ on $N.$

\medskip 
The subset  $Z\subset \rpth$ on which $V$ is zero is the union of 
the eigenspaces of $C.$  
Thus the possibilities for the zero set  $Z$
are:\\ 
(1) For $C_1$  two disjoint projective lines.\\
(2) For $C_2$ one projective line and two points.\\
(3) For $C_3$ four points.\\

\medskip

\begin{lemma}
$C=C_1$ is impossible.
\end{lemma}
\begin{proof} If $C=C_1$ then $Z$ is the union of disjoint two lines $\ell_1,\ell_2$
in $\rpth$ which are invariant under $\hol(\pi_1N).$ Then
$\dev^{-1}(\ell_i)$ is a 1-submanifold in $\tilde{N}$ which is a closed
subset invariant under covering transformations. Hence
$$
\alpha_i\ =\ \pi(\dev^{-1}(\ell_i))
$$ 
is a compact $1$-submanifold in $N.$
Furthermore $\alpha_1\cup\alpha_2$ is the zero set of $v.$ 
We claim 
$\alpha_1\cup\alpha_2$ is not empty; equivalently $v$ must be zero
somewhere in $N.$ Otherwise 
$$
\dev:\tilde{N}\rightarrow X\equiv\rpth\setminus(\ell_1\cup \ell_2)
$$ 
covers an immersion
$$
N\rightarrow X/\hol(\pi_1(N))\cong T^3.
$$ 
This is an immersion of one
closed manifold into another of the same dimension and is thus a
covering map. 
However $N\cong S^1\times S^2$ is not a covering space
of $T^3$ since the latter has universal cover Euclidean space and the
former has universal cover $S^2\times{\mathbb R}.$

Thus we may suppose $\alpha_1$ is not empty.  Let $\beta$ be the
closure of a flowline of $v$ with one endpoint on $\alpha_1.$ Now
$\beta$ is a compact $1$-submanifold of $N$ because its pre-image in
$\tilde{N}$ develops into a closed invariant interval in $\rpth$
with one endpoint in each of $\ell_1$ and $\ell_2.$ Thus $\beta$ has
the other endpoint in $\alpha_2$ which is therefore also non-empty. We
claim that $\alpha_1$ is connected and isotopic in $N=S^2\times S^1$
to $*\times S^1.$ But this is impossible, for $A=S^2\times *$
intersects $\alpha_1$ once transversely. But $A$ lifts to
$\tilde{A}\subset \tilde{N}$ and then $\dev(\tilde{A})$ is an immersion
of a sphere into $\rpth$ which meets $\ell_1=\dev(\pi^{-1}\alpha_1)$ once
transversely. However 
$$
[\ell_1]=0\in H_1(\rpth,{\mathbb Q})
$$ 
and intersection number is an invariant of homology classes, so this is
impossible.

It remains to show $\alpha_1$ is connected and isotopic to $S^1\times *.$ 
Let $\gamma_1$ be a component of $\alpha_1$. Let $U$ be the basin of attraction in $N$ of
$\gamma_1.$ Now $\dev(\pi^{-1}\gamma_1)\subset\ell_1$ and an easy
argument shows these sets are equal. Hence $\dev(\pi^{-1}(U))$ contains
a neighborhood of $\ell_1.$ Thus $U$ contains a small torus transverse
to the flow and bounding a small neighborhood of $\gamma_1.$ Since $U$
is preserved by the flow if follows that $U\cong T^2\times{\mathbb R}.$
The frontier of $U$ in $N$ is contained in $\alpha_1\cup\alpha_2.$
Hence $\alpha_1,\alpha_2$ are both connected and $N=\alpha_1\cup U\cup
\alpha_2.$ Thus $N=H_1\cup H_2$ where 
$$
H_i=\alpha_i\cup T^2\times(0,1]\cong S^1\times D^2.
$$ 
This gives a genus-1 Heegaard
splitting of $N=S^2\times S^1.$ By Waldhausen \cite{W} such a splitting is standard. In
particular this implies that $\alpha_1=\gamma_1$ is isotopic to
$S^1\times*.$ \end{proof}

We are reduced to the case that $C$ is $C_2$ or $C_3.$ In each case
there is a unique isolated zero of $V$ which is a source and another
which is a sink.

\medskip

\begin{lemma}
$\dev(\tilde{N})$ contains no source or sink.
\end{lemma}
\begin{proof} By reversing the flow we may change a source into a sink. So
suppose $p$ is a sink in the image of the developing map.  Let $Q$ be
the projective plane which contains the other points corresponding to
the other eigenspaces of $C.$ Then $Q$ is preserved by $\hol(\pi_1N).$
There is a decomposition into disjoint subspaces $\rpth=p\cup\Omega\cup
Q$ where $\Omega\cong S^2\times {\mathbb R}$ is the basin of attraction
for $p.$ Furthermore each of these subspaces is invariant under
$\hol(\pi_1N).$ Thus there is a corresponding decomposition of $N$ into
disjoint subsets: $\pi(\dev^{-1}(p))$ is a finite non-empty set of
points, $\pi(\dev^{-1}(Q))$ is a compact surface, and
$\pi(\dev^{-1}(\Omega))$ an open submanifold.

Now $\Omega$ admits a foliation by concentric spheres centered on $p$
which is preserved by the flow induced by $V$ and hence by
$\hol(\pi_1N).$ This gives a foliation of $\rpth\setminus p$ by leaves,
one of which is $Q\cong{\mathbb P}^2$ and the others are spheres. Hence
this induces a foliation of $N\setminus \pi(\dev^{-1}(p)).$ Since
$\pi(\dev^{-1}(p))$ is not empty every leaf near it is a small
sphere. Thus $N$ has a singular foliation where the singular points
are isolated and have a neighborhood foliated by concentric
spheres. It follows from the Reeb stability theorem \cite{R} that
if a compact connected $3$-manifold  has a foliation such that each component of the boundary
is a leaf and some leaf
is a sphere, then the manifold is $S^2\times I$ or a punctured
$\rpth.$ But this contradicts that the manifold is $S^2\times S^1$ minus
some open balls.\end{proof}
\medskip

\begin{lemma}\label{claim6}
The flow on $N$ given by $v$ is 
periodic and the flow lines fiber $N$ as a product $S^2\times S^1.$
\end{lemma}
\begin{proof}Let $\lambda$ be the closure of a flowline of $V$ in 
$\rpth$ which has endpoints on the source and sink of $V.$ Such 
flowlines are dense therefore we may choose $\lambda$ to contain a 
point in $dev(\tilde{N}).$ Then $dev^{-1}(\lambda)$ is a non-empty 
closed subset of $\tilde{N}$ which is a $1$-submanifold without 
boundary, since the source and sink are not in $\dev(\tilde{N}).$ 
Hence $\pi(\dev^{-1}\lambda)$ is a compact non-empty $1$-submanifold 
in $N.$ Let $\gamma$ be a component. If $\gamma$ were contractible in 
$N$ then it would lift to a circle in $\tilde{N}$ and be mapped by 
the developing map into $\lambda.$ But this gives an immersion of a 
circle into a  line which is impossible. Thus $[\gamma]\ne0\in\pi_1(N).$

Let $T>0$ be the period of 
the closed flow line $\gamma.$ Let $U$ be the subset of $N$  which is 
the union of closed flow lines of period $T.$ We will show $U$ is 
both open and closed. Since $U$ is not empty and $N$ is connected, 
the claim follows. 

Choose a small disc, $D\subset N,$ transverse to the flow 
and meeting $\gamma$ once. Let $\tilde{D}\subset\tilde{N}$ be a lift 
which meets the component $\tilde{\gamma}\subset\pi^{-1}(\gamma).$ 
The union, $\tilde{Y},$ of the flowlines in $\tilde{N}$ which meets 
$\tilde{D}$ maps homeomorphically by the developing map into a 
foliated neighborhood of the interior of $\lambda.$ Let $\tau$ be the 
covering transformation of $\tilde{N}$ given by 
$[\gamma]\in\pi_1(N).$ Then $\tau$ preserves $dev(\tilde{Y})$ and 
preserves $\tilde{\gamma}$ therefore preserves $\tilde{Y}.$ 
Furthermore 
$$
Y\ =\ \tilde{Y}/\tau\cong dev(\tilde{Y})/hol(\gamma)\ \cong\  S^1\times D^2
$$ 
is foliated as a product. Thus $Y$ is a solid torus 
neighborhood of $\gamma$ in $N$ foliated as a product by flowlines. 
This proves $U$ is open. The limit of flowlines of period $T$ is a 
closed flowline with period $T/n$ for some integer $n>0.$ But $n=1$ 
since the set of flowlines of period $T/n$ is open. Thus $U$ is 
closed.
\end{proof}

\medskip
Let $X=\rpth\setminus Z$ be the subset where $V\ne0.$ Then $X$ is
foliated by flow lines. Let $\Lf$ be the leaf space of the foliation
of $X.$ Then $\Lf$ is a connected $2$-manifold which may be non-Hausdorff.  Since $G$ is
normalized by $\hol(\pi_1M)$ it follows that this group acts on $\Lf.$
Since $\hol(\pi_1N)\subset G$ the action of $\hol(\pi_1N)$ on $\Lf$ is
trivial so the action of $\hol(\pi_1M)$ on $\Lf$ factors through an
action of ${\mathbb Z}_2.$ Thus the holonomy gives an involution on
$\Lf.$ Below we calculate $\Lf$ and this involution in the remaining
cases $C_2,C_3.$

Since $\dev(\tilde{N})\subset X$ there is a map of 
the leaf space of the induced foliation on $\tilde{N}$ into $\Lf.$ By 
Lemma~\ref{claim6} the leaf space of $\tilde{N}$ is the Hausdorff sphere $S^2.$ 
The induced map $h:S^2\rightarrow\Lf$ is a local homeomorphism, which 
we shall call an {\em immersion.} Since $\dev(\tilde{N})\subset \rpth$ 
is invariant under $\hol(\pi_1M)$ it follows that $h(S^2)\subset\Lf$ 
is invariant under the involution. Below we determine all immersions 
of $S^2$ into $\Lf$ and show that  the image is never invariant under 
the involution. This means the remaining cases $C=C_2$ or $C=C_3$ are 
impossible, proving the theorem.

\medskip

\begin{lemma}Case  $C=C_2$ is impossible.
\end{lemma}
\begin{proof} The zero set of $V$ consists of a 
point source, a point sink, 
and a ${\mathbb P}^1$ with hyperbolic dynamics in the transverse direction. 
Every flowline either starts at the source,  or ends at the sink, or does both.  Let $S_1,S_2$ be small spheres around the 
source and sink transverse to the flow. 
The quotient map $X\rightarrow\Lf$ 
embeds each of these spheres, and the union is all of $\Lf.$

It is easy to check that $\Lf$ is obtained from $S_1$ and 
$S_2$ by the following identifications. Regard each sphere as a copy 
of the unit sphere, $S^2,$ in ${\mathbb R}^3.$ Decompose this sphere 
into an equator and northern and southern hemispheres:  
$$
S^2=D_+\cup  E\cup D_-
$$ where 
\begin{align*}
E\ &=\ S^2\cap \{\ x_3=0\ \} \\
D_+\ &=\ S^2\cap \{\ x_3>0\ \} \\ 
D_-\ & =\ S^2\cap\{\ x_3<0\ \}.
\end{align*}

Using the identifications 
of $S_1$ and $S_2$ with $S^2$ identify $D_+\subset S_1$ with 
$D_+\subset S_2$ using the identity map. Identify $D_-\subset S_1$ 
with $D_-\subset S_2$ using the 
map
$(x_1,x_2,x_3)\mapsto(-x_2,x_1,x_3).$

Thus $\Lf$ may be regarded as a sphere with with an extra copy of the
equator.  However one also needs to know a neighborhood basis for the
points on the extra equator. This is determined by the above
description. We show below that every immersed sphere in $\Lf$ is one
of these two embedded spheres. 
The involution swaps $S_1$ and $S_2$
and therefore swaps the two equators in $\Lf$. 
The embedded spheres each contain only one equator and therefore there is no immersion of a sphere into $\Lf$ 
whose image is preserved by the involution.

It remains to determine the possible immersed spheres in $\Lf$. 
There is a decomposition of $\Lf$ into disjoint subsets, 
two of which are the points $(0,0,\pm1)\subset D_{\pm}$ 
and the other subsets are circles which foliate the complement. In particular each of the two equators is a leaf of this foliation.

Suppose $A$ is a sphere and $h:A\rightarrow\Lf$ is an immersion. Then
the pre-images of the decomposition give a decomposition of $A.$ There
are finitely many decomposition elements which are points. Call the
set of these points $P.$ Since $h$ is an immersion, $A\setminus P$ is
decomposed as a $1$-dimensional foliation. Furthermore since $A$ is
compact and the $1$-dimensional leaves in $\Lf$ are closed, their
pre-images in $A$ are compact thus circles.  Thus $A\setminus P$ is
foliated by circles and thus an open annulus. Hence the quotient space
of $A$ corresponding to the decomposition is a closed interval ${\cal
I}\cong[-1,1].$ The endpoints correspond to center type singularities
of a singular foliation on $A.$ The quotient space of the
decomposition of $\Lf$ is a non-Hausdorff interval, ${\cal I}^*\cong
[-1,1]\cup\{0'\},$ with 2 copies of the origin. The endpoints
correspond to the two decomposition elements that are points. The
immersion $h$ induces a map $\overline{h}:{\cal I}\rightarrow {\cal
I}^*.$ Since $h$ is an immersion $\overline{h}$ is also an immersion
(local homeomorphism). Thus $\overline{h}(\pm1)=\pm1.$ The only such
immersion is an embedding which contains one copy of the origin.  This
implies $h$ is an embedding of the form claimed.\end{proof}

It follows from the preceding results that $\dev(\tilde{N})$ is disjoint from the zeroes of the vector field.

\begin{lemma}
Case $C=C_3$  is impossible.
\end{lemma}
\begin{proof} The zero set of $V$ consists of a $4$ points. We label them as
$p_{+++},p_{++-},p_{+--},p_{---}.$ The labelling reflects how many
attracting and how many repelling directions there are. The number of
$-$ signs is the number of attracting directions.  Thus $p_{---}$ is
the sink, $p_{+++}$ is the source.

Every flowline starts at a point with a $+$ label and ends at a point
with a $-$ label. Every ${\mathbb P}^2$ containing three of these four
points is invariant under the flow.

Let $\ell_{-}$ be the ${\mathbb P}^1$ containing $p_{---}$ and $p_{+--}.$
Let $\ell_+$ be the ${\mathbb P}^1$ which contains $p_{+++}$ and
$p_{++-}.$ The restriction of $V$ to each of $\ell_{\pm}$ has on
source and one sink and no other zeroes. There are thus two flowlines
contained in each of $\ell_{\pm}.$

\begin{figure}[ht]	 
\begin{center}
	 \psfrag{T}{$T$}
	 \psfrag{a}{$p_{---}$}
	 \psfrag{d}{$p_{+++}$}
	 \psfrag{b}{$p_{+--}$}
	 \psfrag{w}{$w$}
	 \psfrag{c}{$p_{++-}$}
	 \psfrag{up}{$u_+$}
	 \psfrag{lm}{$\ell_-$}
	 \psfrag{lp}{$\ell_+$}
	 \psfrag{L}{\large${\mathcal L}$}
	 \psfrag{sp}{$S_+$}
		 \includegraphics[scale=0.6]{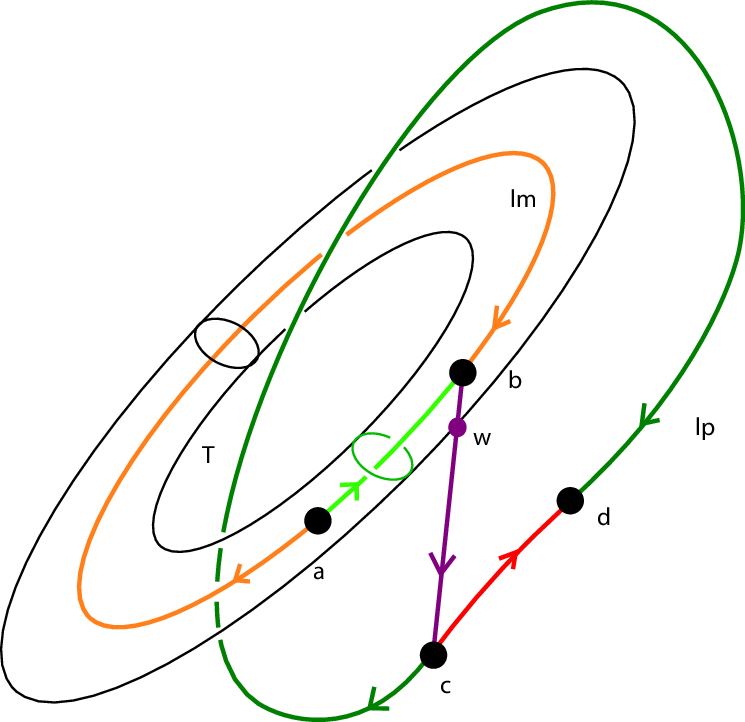}
\end{center}
  \caption{Flowlines for case $C_3$}\label{C3torus}
\end{figure}

Let 
$T$ be a torus transverse to $V$ and which is the boundary of a small 
neighborhood of $\ell_-.$ 
Then $T$ intersects every flowline once 
except the 4 flowlines in $\ell_{\pm}.$ Hence $\Lf$ may be identified with $T$ 
plus $4$ more points. Two of these points come from $\ell_+$ and the 
other two from $\ell_-.$

Since $aca^{-1}=c^{-1}$ it follows that 
$\hol(a)$ conjugates $\hol(c)$ to $\hol(c^{-1})$ and thus $\hol(a)$ 
permutes the zeroes of $V$ by reversing the sign labels. Thus 
$p_{---}\leftrightarrow p_{+++}$ and $p_{+--}\leftrightarrow 
p_{++-}.$

Observe that $T$ can be moved by the flow to a small torus 
around $\ell_+.$ Thus  the involution on $\Lf$ maps the subset 
corresponding to $T$ into itself and swaps the pair of points 
corresponding to $\ell_+$ with the pair corresponding to $\ell_-.$ We 
show below that every immersion of a sphere into $\Lf$ contains 
either the pair of points corresponding to $\ell_+$ or the pair 
corresponding to $\ell_-$ but not both pairs. As before the image of 
the developing map gives an immersion of a sphere into $\Lf$ which is 
preserved by the involution. Thus no such immersion exists and the 
remaining case $C=C_3$ is impossible.

\begin{figure}[ht]	 
\begin{center}
	 \psfrag{T}{$T$}
	 \psfrag{vm}{$v_-$}
	 \psfrag{um}{$u_-$}
	 \psfrag{sm}{$S_-$}
	 \psfrag{w}{$w$}
	 \psfrag{vp}{$v_+$}
	 \psfrag{up}{$u_+$}
	 \psfrag{am}{$\alpha_-$}
	 \psfrag{ap}{$\alpha_+$}
	 \psfrag{L}{\large${\mathcal L}=$}
	 \psfrag{sp}{$S_+$}
		 \includegraphics[scale=0.8]{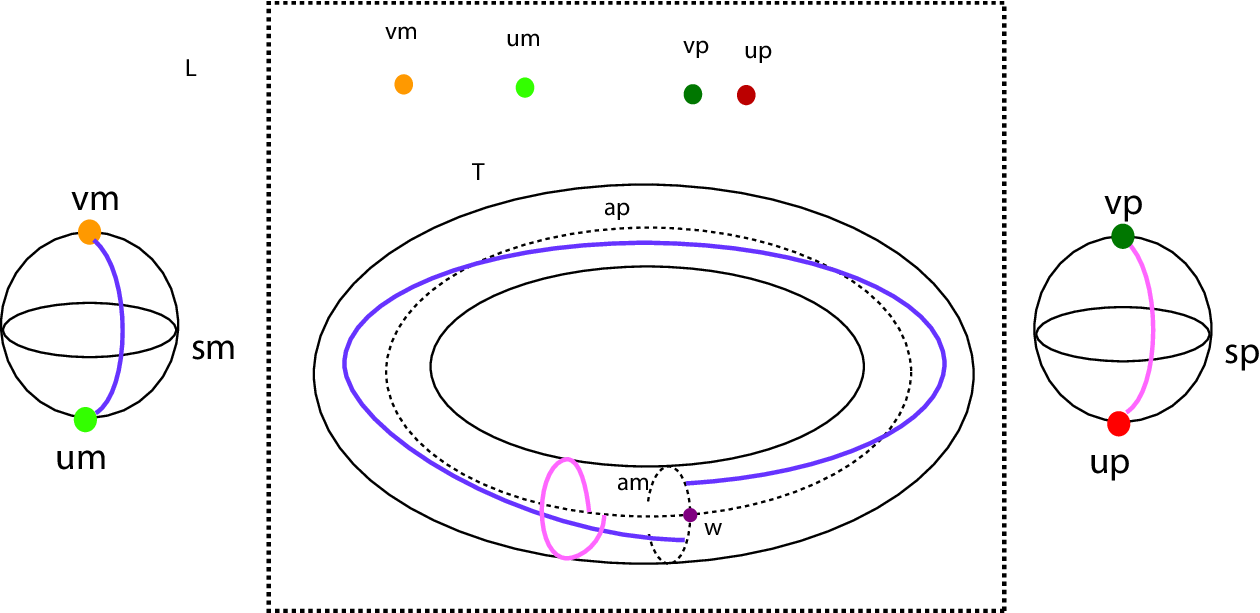}
\end{center}
  \caption{Non-Hausdorff surface ${\mathcal L}$ for case $C_3$}\label{C3new}
\end{figure}

We first describe $\Lf$ in a bit more detail. Let $S_+$ (resp. $S_-$)
be a small sphere around $p_{+++}$ (resp. $p_{---})$ transverse to the
flow. Then every flowline meets $T\cup S_-\cup S_+.$ We next describe
the intersection of the images of $T$ and $S_{\pm}$ in $\Lf.$ We may
choose $S_-$ to be a small sphere inside $T.$ The two flowlines in
$\ell_-$ meet $S_-$ but do not meet $T.$ We call these points $u_-,v_-$ in
$S_{-}$ and the corresponding points in $\Lf$ 
{\em exceptional points.} 
The remaining flowlines that meet $S_-$ intersect $T$ in the
complement of the circle $\alpha_-\subset T$ where $\alpha_-=T\cap
A_-$ and $A_-$ is the ${\mathbb P}^2$ containing the four zeroes of $V$
except $p_{---}.$ A small deleted neighborhood in $\Lf$ of an
exceptional point corresponding to a flowline in $\ell_-$ is an
annulus on one side of $\alpha_-,$ either $\alpha_-\times(0,1)$ or
$\alpha_-\times(-1,0),$ depending on which of the two exceptional
points corresponding to a flowline in $\ell_-$ is chosen.  Similarly
the image of $S_+$ intersects the image of $T$ in the complement of
the circle $\alpha_+= T\cap A_+$ where $A_+$ is the ${\mathbb P}^2$
containing the four zeroes of $V$ except $p_{+++}.$ The circles
$\alpha_-$ and $\alpha_+$ on $T$ meet transversely at a single point $w$
corresponding to the flowline between $p_{+--}$ and $p_{++-}$.

Decompose $\Lf$ into subsets as follows. 
Decompose the image of $T$ by
circles given by a foliation of $T$ by circles parallel to $\alpha_-$
and that are transverse to $\alpha_+.$ The remaining $4$ exceptional
points in $\Lf$ are also decomposition elements.  Let $A$ be a sphere
and $h:A\rightarrow\Lf$ an immersion. As before we deduce that there
is a finite set $P\subset A$ of decomposition elements which are
points. The remaining decomposition elements give a foliation of
$A\setminus P.$ There is a small deleted neighborhood $U\subset
A\setminus p$ of $p\in P$ such that $h(U)$ is an open annulus
$\beta\times(0,1)\subset T$ whose closure consists of two disjoint
circles either parallel to $\alpha_-$ or to $\alpha+_.$ It follows
that the foliation on the subsurface $A_-\subset A$ with these small
open neighborhoods of $P$ removed has the property that each component
of $\partial A_-$ is either transverse to the foliation or is a leaf
of the foliation.  By doubling $A_-$ along the boundary one obtains a
foliation on a closed surface. Hence $A_-$ is an annulus and the
behavior of the foliation on both components of $\partial A_-$ is the
same. If the boundary components are leaves then $h(A)$ contains the
two points corresponding to $\ell_-.$ Otherwise $h(A)$ contains the
two points corresponding to $\ell_+.$ This completes the proof of the
final case, and thus of the theorem.

We remark that the above discussion is similar to the case the
developing map is injective discussed before the proof. 
We argued
above that there is $S^2\subset \tilde{N}\cong S^2\times{\mathbb R}$
immersed in ${\mathbb P}^3$ by the developing map and with the source on
the inside (relative to the flow) and the other three critical points 
on the outside. These three critical points lie on an ${\mathbb P}^2$ which is
preserved by the flow. Indeed they are  the critical points of
a Morse function on this ${\mathbb P}^2$ given by the flow. In {\em some
sense} the proof says this ${\mathbb P}^2$ is {\em outside} the immersed sphere.\end{proof}

\small
\bibliography{refs} 
\bibliographystyle{abbrv}

\address{Department of Mathematics, University of California Santa Barbara, CA 93106, USA}
\email{cooper@math.ucsb.edu}

\address{ Department of Mathematics at the University of Maryland, College Park, MD 20742.} 
\email{wmg@math.umd.edu} 
\Addresses

\end{document}